\documentclass[showkeys,aps,prl,twocolumn,superscriptaddress,10pt,nofootinbib]{revtex4-1}

\usepackage[utf8]{inputenc}  

\usepackage{amsmath}
\usepackage{amsfonts}
\usepackage{amssymb}
\usepackage{amsthm}
\usepackage{mathtools} 
\usepackage{latexsym}

\usepackage{empheq}

\usepackage{wrapfig}
\usepackage{subfigure}

\usepackage{color}

\usepackage[color=lightgray]{todonotes}
   
\usepackage{mathrsfs}
\usepackage{fontenc} 
\usepackage{inputenc} 

\usepackage{verbatim}

\theoremstyle{plain}
\newtheorem{theorem}{Theorem}[section]

\theoremstyle{definition}

\theoremstyle{remark}
\newtheorem{remark}[theorem]{Remark}

\numberwithin{equation}{section} 
\numberwithin{figure}{section}   

\newcommand{\vect}[1]{\mathbf{#1}}

\newcommand{\bu}{\vect{u}}
\newcommand{\bv}{\vect{v}}

\newcommand{\bx}{\vect{x}}

\newcommand{\bomega}{\boldsymbol{\omega}}

\newcommand{\field}[1]{\mathbb{#1}}

\newcommand{\nZ}{\field{Z}}

\newcommand{\nR}{\field{R}}

\newcommand{\maps}{\rightarrow}

\newcommand{\abs}[1]{\left\lvert#1\right\rvert}

\newcommand{\crossprod}{\hspace{-0.3mm}\times}

\newcounter{my_counter}
\newcommand\blfootnote[1]{%
  \begingroup
  \renewcommand\thefootnote{}\footnote{#1}%
  \addtocounter{footnote}{-1}%
  \endgroup
}

\begin{document}


\title[Euler-Voigt Equations]{A Computational Investigation of the Finite-Time Blow-Up of the 3D Incompressible Euler Equations Based on the Voigt Regularization}

\date{\today}
%


 \keywords{Euler-Voigt, Navier-Stokes-Voigt, Inviscid Regularization,Turbulence Models, $\alpha-$Models,\\Blow-Up Criterion for Euler.\\ MSC 2010: 35Q30, 76A10, 76B03, 76D03, 76F20, 76F55, 76F65, 76W05}



\author{Adam Larios}
\affiliation{University of Nebraska-Lincoln}
\email[Adam Larios]{$\quad$alarios@unl.edu}

\author{Mark R. Petersen}
\affiliation{Los Alamos National Lab}
\email[Mark R. Petersen]{$\quad$mpetersen@lanl.gov}

\author{Edriss~S.~Titi}
\affiliation{Texas A\&M University and Weizmann Institute of Science}
\email[Edriss~S.~Titi]{$\quad$titi@math.tamu.edu \\ and edriss.titi@weizmann.ac.il}

\author{Beth Wingate}{
\affiliation{University of Exeter}
\email[Beth Wingate]{$\quad$B.Wingate@exeter.ac.uk}


\date{\today}

\begin{abstract}
We report the results of a computational investigation of two blow-up criteria for the 3D incompressible Euler equations.  One criterion was proven in a previous work, and a related criterion is proved here.  These criteria are based on an inviscid regularization of the Euler equations known as the 3D Euler-Voigt equations, which are known to be globally well-posed.  Moreover, simulations of the 3D Euler-Voigt equations also require less resolution than simulations of the 3D Euler equations for fixed values of the regularization parameter $\alpha>0$.  Therefore, the new blow-up criteria allow one to gain information about possible singularity formation in the 3D Euler equations indirectly; namely, by simulating the better-behaved 3D Euler-Voigt equations.  The new criteria are only known to be sufficient for blow-up.  Therefore, to test the robustness of the inviscid-regularization approach, we also investigate analogous criteria for blow-up of the 1D Burgers equation, where blow-up is well-known to occur.
\end{abstract}

\maketitle

\section{Introduction\label{sec:intro}}
The 3D Euler equations for incompressible inviscid fluid flow are a source of much mathematical and scientific interest.  In particular, these equations exhibit many of the same  difficulties as the 3D Navier-Stokes equations in the case of large Reynolds numbers.  The question of whether these equations develop a finite-time singularity remains an extremely challenging open problem.  \blfootnote{{\bf To appear in:} Theoretical and Computational Fluid Dynamics.}

A blow-up criterion for the 3D Euler equations for ideal incompressible flow was reported in \cite{Larios_Titi_2009}.  This criterion is of a different character than, e.g., the well-known Beale-Kato-Majda criterion \cite{Beale_Kato_Majda_1984}.   Traditional computational searches for blow-up seek to identify singularities by analyzing the vorticity coming from the 3D Euler equations themselves, which are not known to be globally well-posed, and moreover, are extremely difficult to simulate accurately.  In contrast, the blow-up criterion in \cite{Larios_Titi_2009} only relies on analyzing the vorticity of the 3D Euler-Voigt equations, which are globally well-posed and can be less computationally intensive to simulate accurately.  

An important aspect of the Euler-Voigt model, when used as a regularization for the Euler equations, is that the regularization is inviscid in the sense that it does not add artificial viscosity.  Hence, we refer to the Voigt-regularization as an \textit{inviscid regularization}.  Moreover, the Voigt-regularization can be used to stabilize simulations of the Euler equations by a method different from adding artificial viscosity, as is done, e.g., in LES (Large-Eddy Simulation) models (see, e.g., \cite{Berselli_Iliescu_Layton_2006_book}, and the references therein).  Inviscid regularization is distinct from regularizations that use artificial viscosity: while artificial viscosity removes energy from the system, the Euler-Voigt equations conserve a modified energy for all time (see \eqref{energy_equality} below). 
We use this conservation as one test of the validity of our simulations.  
Moreover, the blow-up criterion we test is derived from \eqref{energy_equality} and the short-time energy conservation of the 3D Euler equations.

%

In this article, we describe the first computational search for blow-up of the 3D Euler equations based on a Voigt-type blow-up criterion.  We also provide a new blow-up criterion that is similar in character to the criterion in \cite{Larios_Titi_2009}, but that has several advantages over it.  
One interesting result of the present work is that extrapolation to 
$\alpha = 0$ suggests  the development of a singularity in the 3D Euler 
equations.
The blow-up time $T_*$ coincides approximately with the prediction $T_*\approx4.4\pm0.2$ in \cite{Brachet_Meiron_Orszag_Nickel_Morf_Frisch_1983_JFM} (see also \cite{Bustamante_Brachet_2012}).  However, the purpose of this work is chiefly to motivate the fluid mechanics computational community toward further investigation of this type of criterion, rather than to make a definite claim about blow-up.  Because this is a new approach to studying blow-up, we show how the method provides evidence for blow-up in a case where blow-up is well understood; namely, in the inviscid Burgers equation.  For additional corroboration of the method, we also show that blow-up is not detected in the viscous Burgers equation, where it is known that blow-up does not occur.


The Euler-Voigt equations were proposed as an inviscid regularization of the Euler equations in \cite{Cao_Lunasin_Titi_2006}, where they were first studied.  Their viscous counterpart, called the Navier-Stokes-Voigt equations, were studied much earlier in \cite{Oskolkov_1973,Oskolkov_1982}.  The Euler-Voigt equations are given by
\begin{subequations}\label{EV}
\begin{empheq}[left=\empheqlbrace]{align}
\label{EV_mo}
-\alpha^2\partial_t\nabla^2\bu
+\partial_t\bu +(\bu\cdot\nabla)\bu
+\nabla p&=0
,
\\
\label{EV_div}
\nabla\cdot \bu
&= 0,
\\
\label{EV_int}
\bu(\bx,0) &= \bu_0(\bx).
\end{empheq}
\end{subequations}
Here $\alpha>0$ is a regularization parameter having units of length.  Note that the usual incompressible Euler equations are formally obtained by setting $\alpha=0$.  The unknowns are the fluid velocity field $\bu(\bx,t)=(u_1,u_2,u_3)$, and the fluid pressure $p(\bx,t)$, where $\bx = (x_1,x_2,x_3)$, and $t\geq0$. In the present work, we consider only the case of periodic boundary conditions.  (Periodic boundary conditions are often used in computational studies; the review \cite{Gibbon_2008} cites more than twenty such studies.)  Without loss of generality, we also assume that $\int_\Omega\bu_0(\bx)\,d\bx=0$, which with \eqref{EV_mo} and \eqref{EV_div} implies $\int_\Omega\bu(\bx,t)\,d\bx=0$ for all $t$.  We denote by $\bu^\alpha$ the solution to \eqref{EV}, and by $\bu$ a  solution to the Euler equations, both starting from the same initial condition $\bu_0$.  In addition, we denote the corresponding vorticities $\bomega := \nabla\crossprod\bu$, and also $\bomega^\alpha := \nabla\crossprod\bu^\alpha$.


 System \eqref{EV} was introduced in \cite{Cao_Lunasin_Titi_2006}, where existence and uniqueness of solutions was proven for all times $t\in(-\infty,\infty)$.  The Euler-Voigt and Navier-Stokes-Voigt equations have been studied analytically and extended in a wide variety of contexts (see, e.g., \cite{Bohm_1992,Catania_2009,Catania_Secchi_2009,Cao_Lunasin_Titi_2006,Larios_Titi_2009,Larios_Lunasin_Titi_2015,Ebrahimi_Holst_Lunasin_2012,Levant_Ramos_Titi_2009,Khouider_Titi_2008,Olson_Titi_2007,Oskolkov_1973,Oskolkov_1982,Ramos_Titi_2010,Kalantarov_Levant_Titi_2009,Kalantarov_Titi_2009}, and the references therein).  The first computational study of the Navier-Stokes-Voigt and MHD-Voigt equations was carried out in \cite{Kuberry_Larios_Rebholz_Wilson_2012}.  A recent computational study \cite{DiMolfetta_Krstlulovic_Brachet_2015} studied the energy spectrum and other properties of the Euler-Voigt equations.  Energy decay for Navier-Stokes-Voigt was studied in \cite{Layton_Rebholz_2013_Voigt}. 

 In \cite{Cao_Lunasin_Titi_2006}, an ``$\alpha$-energy'' equality was proved to hold for solutions of \eqref{EV} for all $t\in\nR$, namely,
 \begin{align}\label{energy_equality}
   E_\alpha(t):= \|\bu^\alpha(t)\|_{L^2}^2 +\alpha^2\|\nabla\bu^\alpha(t)\|_{L^2}^2
   = E_\alpha(0).
\end{align}
One aim of this paper is to investigate the connection between the Euler equations and Euler-Voigt equations as $\alpha\maps0$.  In \cite{Larios_Titi_2009}, it was shown that, for sufficiently smooth initial data, on the time interval $[0,T]$ of existence and uniqueness for strong solutions of the Euler equations, the following estimate
holds:
\begin{align}
\label{conv_claim_full}
&\|\bu(t)-\bu^\alpha(t)\|_{L^2}^2
+\alpha^2\|\nabla(\bu(t)- \bu^\alpha(t))\|_{L^2}^2
\\&\leq\notag
C\alpha^2(e^{Ct}-1),
\end{align}
where the constant $C$ depends on $\|\bu\|_{L^\infty(0,T;H^3)}$.
In particular, as $\alpha\maps0$, solutions to \eqref{EV} converge to the solution the Euler equations in the $L^\infty([0,T];L^2)$ norm at a rate no worse than $\mathcal{O}(\alpha)$.  Combining this with \eqref{energy_equality} and the equality $\|\bu(t)\|_{L^2} = \|\bu_0\|_{L^2}$, which holds on $[0,T]$, it was proved in \cite{Larios_Titi_2009}, by contradiction, that if 
\begin{align}\label{blow_up_criterion_old}
\sup_{t\in[0,T^{*}]}\limsup_{\alpha\maps0^+}\,(\alpha\|\nabla \bu^\alpha(t)\|_{L^2})>0,
\end{align}
then the 3D Euler equations must develop a singularity at or before time $T^*$.  
We shall show in Section \ref{sec_new_criterion} that if
\begin{align}\label{blow_up_max_first}
\limsup_{\alpha\maps0^+}\Big(\alpha\sup_{t\in[0,T^{*}]}\|\nabla \bu^\alpha(t)\|_{L^2}\Big)>0,
\end{align}
then again the 3D Euler equations must develop a singularity at or before time $T^*$. 
As noted below, \eqref{blow_up_criterion_old} implies \eqref{blow_up_max_first}, and hence 
\eqref{blow_up_max_first} is a stronger criterion than \eqref{blow_up_criterion_old}, i.e., singularities indicated by \eqref{blow_up_criterion_old} 
will also be
indicated by \eqref{blow_up_max_first}.  

\begin{remark}\textit{Comparison with original criterion.}
\noindent
The new blow-up criterion \eqref{blow_up_max_first}  is stronger than  \eqref{blow_up_criterion_old}, 
since, for any $u^\alpha\in C([0,T],L^2)\cap L^1([0,T],H^1)$,
\begin{align}
\sup_{t \in [0,T]} \alpha \| \nabla u^\alpha (t) \|_{L^2} \ge \alpha \| \nabla u^\alpha (t) \|_{L^2},
\end{align}
for any $t\in[0,T]$, so we may take the $\limsup_{\alpha \to 0^+}$ of both sides to obtain
\begin{align}
\limsup_{\alpha \to 0^+} \sup_{t \in [0,T]} \alpha \| \nabla u^\alpha (t) \|_{L^2} \ge \limsup_{\alpha \to 0^+} \alpha \| \nabla u^\alpha (t) \|_{L^2}.
\end{align}
The left-hand side is constant, and the right-hand side depends on t. Thus,
\begin{align}
\limsup_{\alpha \to 0^+} \sup_{t \in [0,T]} \alpha \| \nabla u^\alpha (t) \|_{L^2} \ge  \sup_{t \in [0,T]} \limsup_{\alpha \to 0^+} \alpha \| \nabla u^\alpha (t) \|_{L^2}.
\end{align}
Therefore, if the right-hand side is positive, the left-hand side is positive.  Hence, \eqref{blow_up_criterion_old} implies \eqref{blow_up_max_first}.
\end{remark}


The computational search for blow-up has a rich recent history, see, e.g. \cite{Brachet_Meiron_Orszag_Nickel_Morf_Frisch_1983_JFM,Brachet_Meiron_Orszag_Nickel_Morf_Frisch_1984_JSP,Bustamante_2011,Cichowlas_Brachet_2005,Deng_Hou_Yu_2005,Grafke_Grauer_2013_Lagrangian_no_blow_up,Grafke_Homann_Dreher_Grauer_2007,Hou_2009,Hou_Lei_Luo_Wang_Zou_2014_Euler,Hou_Li_2006_Depletion,Hou_Li_2008_Blowup,Hou_Li_2008_Numerical,Kerr_1993,Kerr_2013_Euler_Blowup,Luo_Hou_2013_Potentially_Singular,Luo_Hou_2014_Euler_BlowUp,Pouquet_Brachet_Mininni_2010_TG_MHD,Siegel_Caflisch_2009} and the references therein.  Since it is unknown whether the 3D Euler equations become singular in a finite interval of time, several criteria for the blow-up of solutions have arisen in the literature, e.g., \cite{Beale_Kato_Majda_1984,Cheskidov_Shvydkoy_2014_unified_blow_up,Cheskidov_Shvydkoy_2014_unified_blow_up,Constantin_Fefferman_1993,Constantin_Fefferman_Majda_1996,Ferrari_1993,Gibbon_Titi_2013_Blowup,Kozono_Ogawa_Taniuchi_2002,Ponce_1985}. 
Perhaps the most celebrated is the Beale-Kato-Majda criterion \cite{Beale_Kato_Majda_1984} which states that the solution is non-singular on $[0,T]$ if and only if 
\begin{align}\label{BKM}
\int_0^T\|\bomega(t)\|_{L^\infty}\,dt <\infty.
\end{align}
Hence, in many computational searches for blow-up of solutions of the Euler equations (see, e.g., \cite{Deng_Hou_Yu_2005,Hou_2009,Hou_Li_2008_Blowup,Hou_Li_2008_Numerical,Kerr_1993,Kerr_2013_Euler_Blowup}, and references therein), $\|\bomega(t)\|_{L^\infty}$ is the main quantity of interest.  Thanks to the identity $\|\nabla\bv\|_{L^2} = \|\nabla\crossprod\bv\|_{L^2}$, holding for all smooth divergence-free functions $\bv$, one can view \eqref{blow_up_criterion_old} and \eqref{blow_up_max_first} as conditions on the vorticity $\bomega^\alpha$ of the Euler-Voigt equations.   In Fig. \ref{fig_spec}, we plot the time evolution of the $L^2$ energy spectrum, which is captured within an accuracy of $10^{-12}$.

\begin{figure}[htp]
\includegraphics[width=0.98\columnwidth,trim = 0mm 0mm 0mm 0mm, clip]{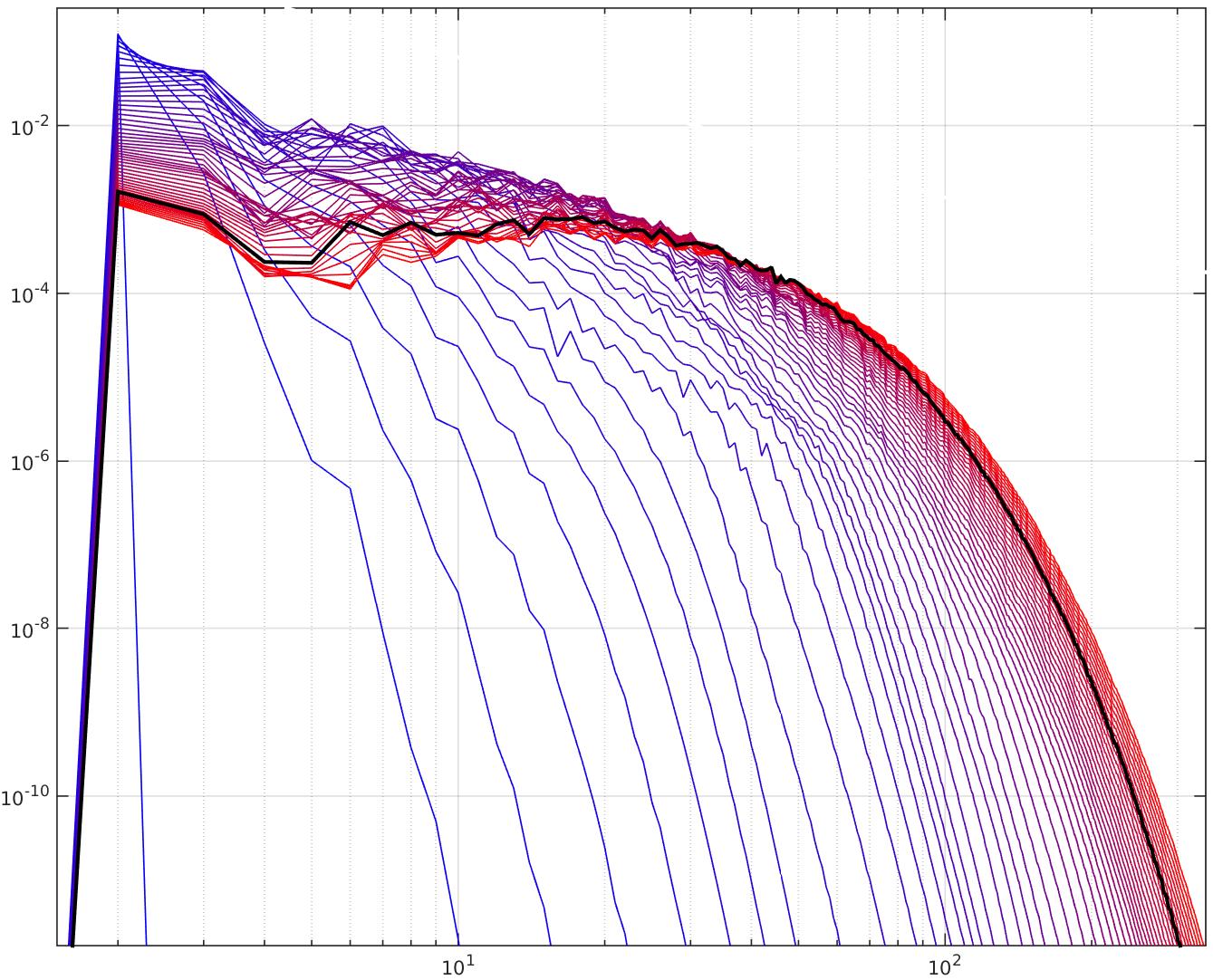}

\vspace{-0.3cm}

\caption{The $L^2$ spectrum vs. wave number the solution of the Euler-Voigt equations with $\alpha=12/1024$ at times $t = 0.0, 0.1, \ldots, 4.9, 5.0$.  At $t=0.0$, the spectrum is blue.  It becomes increasing red as time evolves.  The black spectrum corresponds to time $4.2$, where the smallest slope is observed in Fig. \ref{fig_EV_max_blow_up}.  Resolution: $N^3=1024^3$.  
}\label{fig_spec}
\end{figure}



\begin{remark}
 We emphasize that quantity \eqref{BKM} is computed from solutions of the 3D Euler equations, which are not known to be globally well-posed. In contrast, the quantity $\|\nabla \bu^\alpha\|_{L^2}$ in \eqref{blow_up_criterion_old} and \eqref{blow_up_max_first} is computed from solutions to \eqref{EV}, which is known to be well-posed globally in time.  This gives a mathematical foundation for reliably computing $\|\nabla\bu^\alpha\|_{L^2}$.  Moreover, due to \eqref{energy_equality}, the growth of the gradient---and hence the development of small length scales---is limited. This is important in numerical simulations, where one has only finite resolution.  In contrast, the 3D Euler equations are not known to possess such a quality.
\end{remark}

In Section \ref{sec_new_criterion}, we improve criterion \eqref{blow_up_criterion_old} to criterion \eqref{blow_up_max_first}.  Numerical methods are described in Section \ref{sec_num_methods}.  The main work is in Section \ref{sec_blow_up}, where we computationally investigate the dependence of $\|\nabla \bu^\alpha(t)\|_{L^2}$ on $\alpha$ and $t$, for some given initial data, as $\alpha\maps0$.   It is unknown whether \eqref{blow_up_max_first} (or \eqref{blow_up_criterion_old}) is a necessary condition for the blow-up of solutions of the 3D Euler equations. Hence, to further support the notion that blow-up may be indicated by \eqref{blow_up_max_first}, we consider the 1D inviscid Burgers equation, which  is well-known to have solutions that blow up in finite time.  In Section \ref{sec_BBM}, we apply a Voigt-type regularization to the 1D Burgers equation (yielding the Benjamin-Bona-Mahoney (BBM) equation \eqref{BBM}), and show computationally that the analogues of \eqref{blow_up_criterion_old} and \eqref{blow_up_max_first} appear to be satisfied when $T^*$ approaches the blow-up time of the Burgers equation.  Moreover, we show that \eqref{blow_up_max_first} is no longer satisfied after the addition of viscosity, which conforms with the global-well-posedness of the viscous Burgers equation.

\section{Improved Blow-Up Criterion}\label{sec_new_criterion}

In this section, we improve blow-up criterion \eqref{blow_up_criterion_old} to blow-up criterion \eqref{blow_up_max_first}.  Both criteria are derived from \eqref{energy_equality} and the short-time energy conservation of the 3D Euler equations; hence, we briefly discuss recent work relating energy conservation to smoothness.

We denote by $L^p$ and $H^s$ the usual Lebesgue and Sobolev spaces over the periodic domain $\Omega\equiv[0,1]^3:=\nR^3/\nZ^3$, respectively.  It is a classical result (see, e.g., \cite{Majda_Bertozzi_2002,Marchioro_Pulvirenti_1994}) that, for initial data $\bu_0\in H^3$ satisfying $\nabla\cdot\bu_0=0$, a unique strong solution $\bu$ of the 3D Euler equations exists and is unique on a maximal time interval that we denote by $[0,T^*)$.  Moreover, one has
\begin{align}\label{Euler_energy_equality}
\|\bu(t)\|_{L^2} = \|\bu_0\|_{L^2} \text{ on } [0,T^*).
\end{align}
Equation \eqref{Euler_energy_equality} holds under weaker conditions on the smoothness of the solutions of the 3D Euler equations, as it was conjectured by Onsager (see, e.g., \cite{Cheskidov_Constantin_Friedlander_Shvydkoy,Constantin_E_Titi_1994,Eyink_1994,Eyink_Sreenivasan_2006,Onsager_1949}).  However, the existence of such weak solutions for arbitrary admissible initial data is still out of reach.  In \cite{Bardos_Titi_2010}, it was shown that a certain class of shear flows are weak solutions in $L^\infty((0,T);L^2)$ that conserve energy.  Furthermore, families of weak solutions that do not satisfy the regularity assumed in the Onsager conjecture have been constructed that do not satisfy \eqref{Euler_energy_equality}, see, e.g.,  \cite{Buckmaster_DeLellis_Isett_Szekelyhidi_2015,Buckmaster_2015,DeLellis_Szekelyhidi_2010_admissibility,DeLellis_Szekelyhidi_2013,DeLellis_Szekelyhidi_2014_Onsager,Isett_2013_thesis}.

The following theorem was proved in \cite{Larios_Titi_2009}.  It is based on a similar theorem for the surface quasi-geostrophic (SQG) equations in \cite{Khouider_Titi_2008}.
\begin{theorem}[\cite{Larios_Titi_2009}]\label{thm_blow_up_old}
Assume $\bu_0\in H^s$, for some $s\geq3$, with $\nabla\cdot\bu_0=0$. Suppose there exists a $T >0$ such that solutions $\bu^\alpha$ of \eqref{EV} satisfy \eqref{blow_up_criterion_old}. 
Then the 3D Euler equations, with initial data $\bu_0$, must develop a singularity within the interval $[0,T]$.
\end{theorem}

A technical difficulty arises in computational tests of Theorem \ref{thm_blow_up_old}. Mathematically, one may imagine fixing a $t>0$ and computing 
\begin{align}\label{blow_up_quantity_1}
\limsup_{\alpha\maps0^+}\left(\alpha\|\nabla \bu^\alpha(t)\|_{L^2}\right).
\end{align}
However, computationally, it is more natural to first fix $\alpha>0$ as a parameter, and then to compute $\bu^\alpha(t)$ as $t$ increases up to a time $T$ (e.g., by a standard time-stepping method).  Therefore, to construct curves of $\alpha\|\nabla \bu^\alpha(t)\|_{L^2}$ vs. $\alpha$ for each fixed $t$, one must jump from solution to solution as $\alpha$ varies.  This gives rise to some of the technical issues discussed above. However, suppose for a moment that one is allowed to commute the two limiting operations in \eqref{blow_up_criterion_old}.  One would then obtain criterion \eqref{blow_up_max_first}.  The quantity in \eqref{blow_up_max_first} is arguably easier to track, as discussed above.
It is the purpose of this section to show rigorously that \eqref{blow_up_max_first} implies that the 3D Euler equations develop a singularity within the interval $[0,T]$.

Let $T>0$ be given.  Assume that a given solution to the Euler equations is smooth on $[0,T]$, so that in particular, \eqref{Euler_energy_equality} holds.
We emphasize that \eqref{Euler_energy_equality} depends on the regularity of the 3D Euler equations, and if a finite-time singularity develops, \eqref{Euler_energy_equality} might not hold.

\begin{theorem}
Let $\bu_0\in H^s$, $s\geq3$, with $\nabla\cdot\bu_0=0$, and let $\bu^\alpha$ be the corresponding unique solution of \eqref{EV}.  Suppose that \eqref{blow_up_criterion_old} holds
for some $T>0$.  Then the Euler equations must develop a singularity within the interval $[0,T]$.
\end{theorem}

\begin{proof}
We prove the contrapositive.  Assume that $\bu$ is a solution of the 3D Euler equations, with initial data $\bu_0\in H^s$, $s\geq3$, that remains smooth on the interval $[0,T]$.  In particular, the smoothness implies that \eqref{Euler_energy_equality} holds.  From \eqref{conv_claim_full}, for any $t\in[0,T]$, it follows that
\begin{align}\label{conv_claim}
\|\bu^\alpha(t)\|_{L^2}
&\geq
\|\bu(t)\|_{L^2}-C\alpha(e^{Ct}-1)^{1/2}
\\&\geq
\|\bu(t)\|_{L^2}-C\alpha(e^{CT}-1)^{1/2}
\\&=\nonumber
\|\bu_0\|_{L^2}-C\alpha(e^{CT}-1)^{1/2}.
\end{align}
Here, we have used \eqref{Euler_energy_equality}.  
Let $\alpha>0$ be sufficiently small so that the right-hand side is positive (e.g., choose, $\alpha< \|\bu_0\|_{L^2}/(C(e^{CT}-1)^{1/2})$.  Squaring, we obtain,
\begin{align}\label{L2est}
\|\bu^\alpha(t)\|_{L^2}^2
&\geq
\|\bu_0\|_{L^2}^2
-2C\alpha\|\bu_0\|_{L^2}(e^{CT}-1)^{1/2}
 \\&\qquad\notag
+C^2\alpha^2(e^{CT}-1).
\end{align}
Combining \eqref{L2est} and \eqref{energy_equality}, we discover
 \begin{align*}
\alpha^2\|\nabla\bu^\alpha(t)\|_{L^2}^2
&\leq  
\alpha^2\|\nabla\bu_0\|_{L^2}^2
+2C\alpha\|\bu_0\|_{L^2}(e^{CT}-1)^{1/2}
 \\&\qquad
-C^2\alpha^2(e^{CT}-1).
\end{align*}
Thus,
$
 \limsup_{\alpha\maps0^+}\sup_{t\in [0,T]} \alpha^2\|\nabla\bu^\alpha(t)\|_{L^2}^2=0,
$
which contradicts assumption \eqref{blow_up_criterion_old}, and therefore the solution $\bu$ of the Euler equations must become singular within the interval $[0,T]$.
\end{proof}

\section{Numerical Methods}\label{sec_num_methods}
All simulations were carried out using a pseudospectral method on the periodic unit cube; namely, with derivatives computed in Fourier space, and products computed in physical space with the $2/3$'s dealiasing rule applied.  Time stepping for the inviscid equations was done using a fully-explicit fourth-order Runge-Kutta-4 scheme complying with the advective CFL condition.  (For the viscous Burgers equation, an integrating-factor method adapted to Runge-Kutta-4 was used to avoid the viscous CFL restriction.)  The pressure was computed explicitly by the standard Chorin-Temam projection method \cite{Chorin_1968,Temam_1969_projection}.  For the Euler-Voigt simulations, Taylor-Green initial data was used on the domain $[0,1]^3$, namely,
\begin{align}
\begin{split}\label{TaylorGreen}
   u_1  &=  \phantom{-}\sin(2\pi x)\cos(2\pi y)\cos(2\pi z), 
\\ u_2  &= -\cos(2\pi x)\sin(2\pi y)\cos(2\pi z),
\\ u_3 &=  0.
\end{split}
\end{align}
This choice of initial data is very commonly used in computational studies of blow-up for the 3D Euler equations.  See, e.g., \cite{Brachet_Meiron_Orszag_Nickel_Morf_Frisch_1984_JSP,Brachet_Meiron_Orszag_Nickel_Morf_Frisch_1983_JFM}.


It is important for this study that the energy and the enstrophy are properly captured.  Therefore, 
we consider the maximum relative error in the $\alpha-$energy by
\begin{align*}
\varepsilon_\text{rel}:=
\max_{t\in[0,T]}
    \abs{\frac{E_\alpha(t) - E_\alpha(0)}{E_\alpha(0)}}.
\end{align*}

Due to the  Runge-Kutta-4 time stepping, perfect $\alpha-$energy conservation is not expected.  However, every Euler-Voigt simulation at resolution $1024^3$ and $512^3$ reported in this article had $\varepsilon_\text{rel}<2.2\times10^{-11}$ over the time interval of integration.  For the inviscid BBM simulations, $\varepsilon_\text{rel}<2.4\times10^{-14}$.  For the viscous BBM simulations, $\varepsilon_\text{rel}<2.8\times10^{-13}$ (for the viscous simulations the definition of $E_\alpha(t)$ was adapted to include the term $2\nu\int_0^t\|u_x(s)\|_{L^2}\,ds$, computed using Runge-Kutta-4 integration). In Fig. \ref{fig_energy_equality}, one can see the typical behavior of the terms comprising the $\alpha$-energy $E_\alpha(t)$, with a transfer of the energy ($\|\bu^\alpha\|_{L^2}^2$) to the scaled enstrophy ($\alpha^2\|\nabla\bu^\alpha\|_{L^2}^2$). 

\begin{figure}[htp]
\subfigure[$\alpha = 4/256$]{
\includegraphics[width=0.32\textwidth]{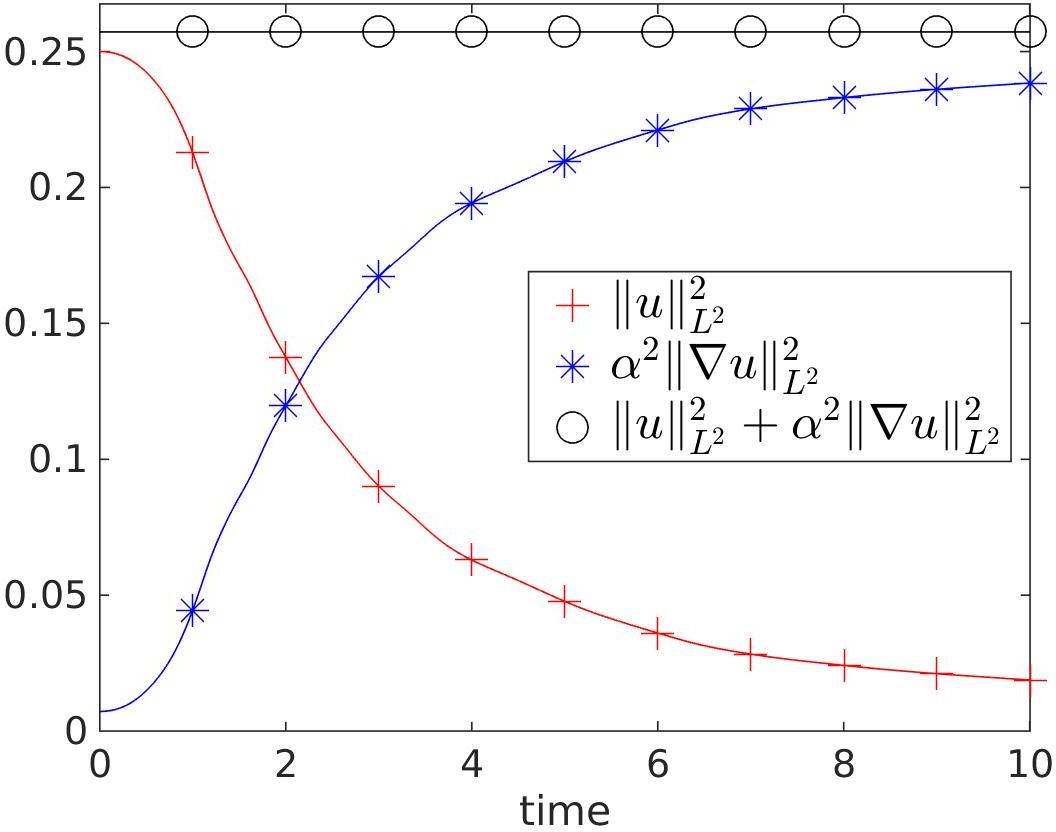}}
\subfigure[$\alpha = 2/256$]{
\includegraphics[width=0.32\textwidth]{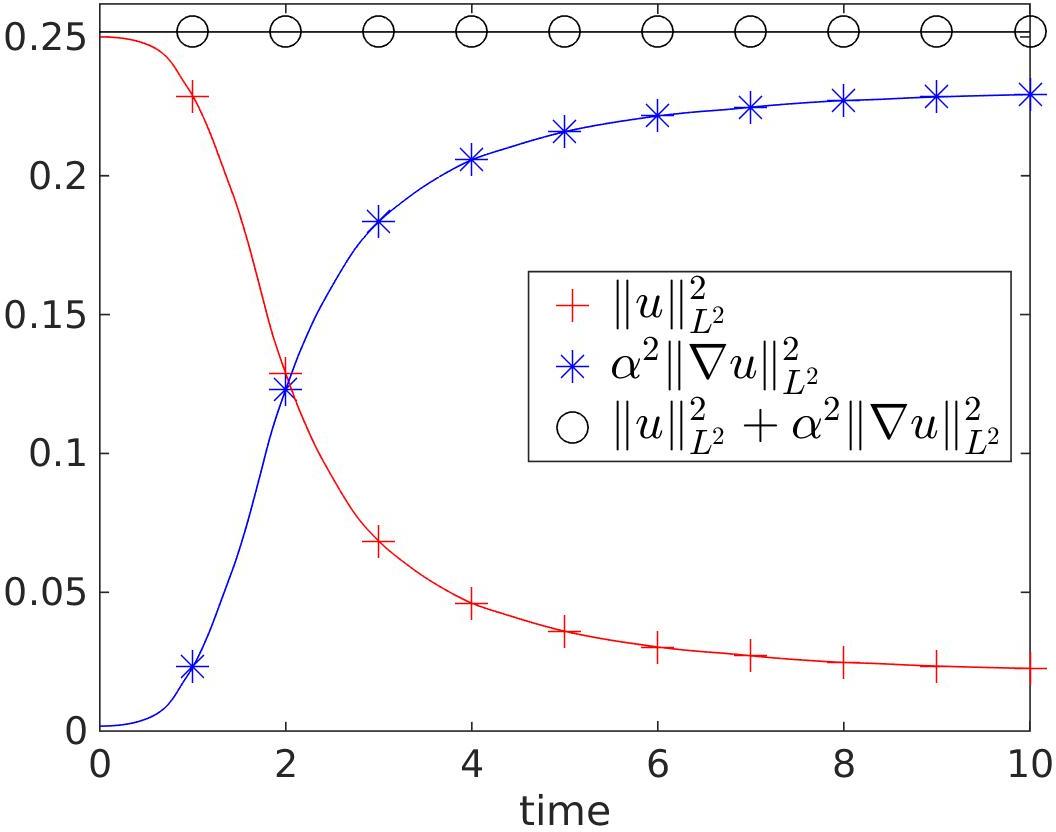}}
\subfigure[$\alpha = 1/256$]{
\includegraphics[width=0.32\textwidth]{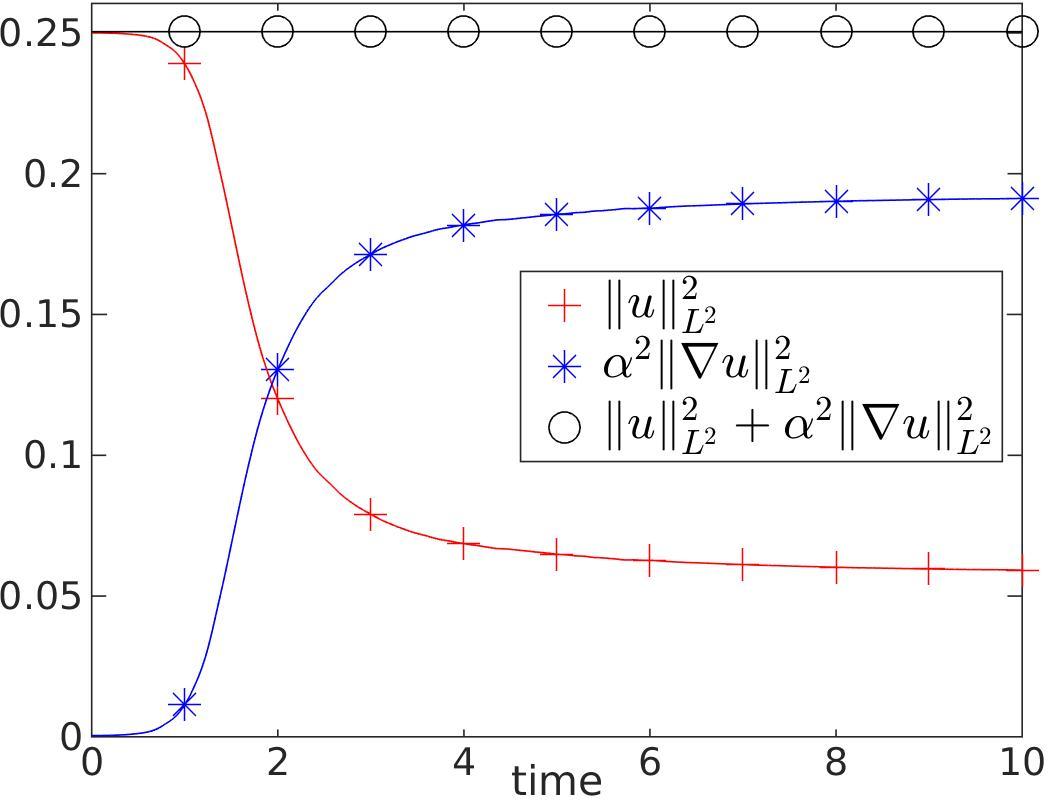}}
\caption{Energy and enstrophy (scaled by $\alpha^2$) vs. time for the 3D Euler-Voigt equations.  
(red ``$+$'': $\|\bu^\alpha(t)\|_{L^2}^2$, 
blue ``\textasteriskcentered'': $\alpha^2\|\nabla\bu^\alpha(t)\|_{L^2}^2$, 
black ``$\circ$'': $\|\bu^\alpha(t)\|_{L^2}^2 +\alpha^2\ \nabla\bu^\alpha(t)\|_{L^2}^2$.)    Resolution: $256^3$. 
}
\label{fig_energy_equality}
\end{figure}

\begin{remark} We emphasize that, since \eqref{EV} is globally well-posed in time, we are allowed to integrate the equations \textit{beyond the point of possible singularity} for the 3D Euler equations.  That is, if the Euler equations develop a singularity at time $T^{*}$, for given initial data, we may safely integrate \eqref{EV} with the same initial data up to and beyond $T^{*}$.  We believe this to be a major distinction of the blow-up criteria \eqref{blow_up_criterion_old} and \eqref{blow_up_max_first} from other blow-up criteria for the 3D Euler equations, such as \eqref{BKM}.
\end{remark}


\section{Singularity Detection}\label{sec_blow_up}
In this section, we computationally investigate the blow-up criterion \eqref{blow_up_max_first}.  We simulate solutions of \eqref{EV} with initial data \eqref{TaylorGreen}, tracking the quantity 
\begin{align}\label{blow_up_no_max}
   \|\nabla\bu^\alpha(t)\|_{L^2}
   \equiv
   \|\bomega^\alpha(t)\|_{L^2},
\end{align}
for several values of $t$, as $\alpha\maps0$, shown in Fig. \ref{fig_EV_max_blow_up} as contours of constant $t$.  
%
%
%
%
%
%
Let us make the ansatz that
\begin{align}\label{sup_big_oh}
 \sup_{t\in[0,T^*]}\|\nabla\bu^\alpha(t)\|_{L^2}\sim \mathcal{O}(\alpha^p),
\end{align}
for $T^*>0$ sufficiently large and for some power $p$.
If $p\leq-1$, then \eqref{blow_up_max_first} holds, and the Euler equations develop a singularity within the interval $[0,T^*]$.  The quantity in \eqref{sup_big_oh} is shown in Fig. \ref{fig_EV_max_blow_up} as a function of $\alpha$ with various values of $T^*$. 
The slope of the lines corresponding to $T^* \approx 4.2$ are strictly less than $-1$ for small $\alpha$, indicating  a possible blow-up of the Euler equations near time $T^* \approx 4.2$.



\begin{figure}[htp]
\includegraphics[width=0.98\columnwidth,trim = 0mm 0mm 0mm 0mm, clip]{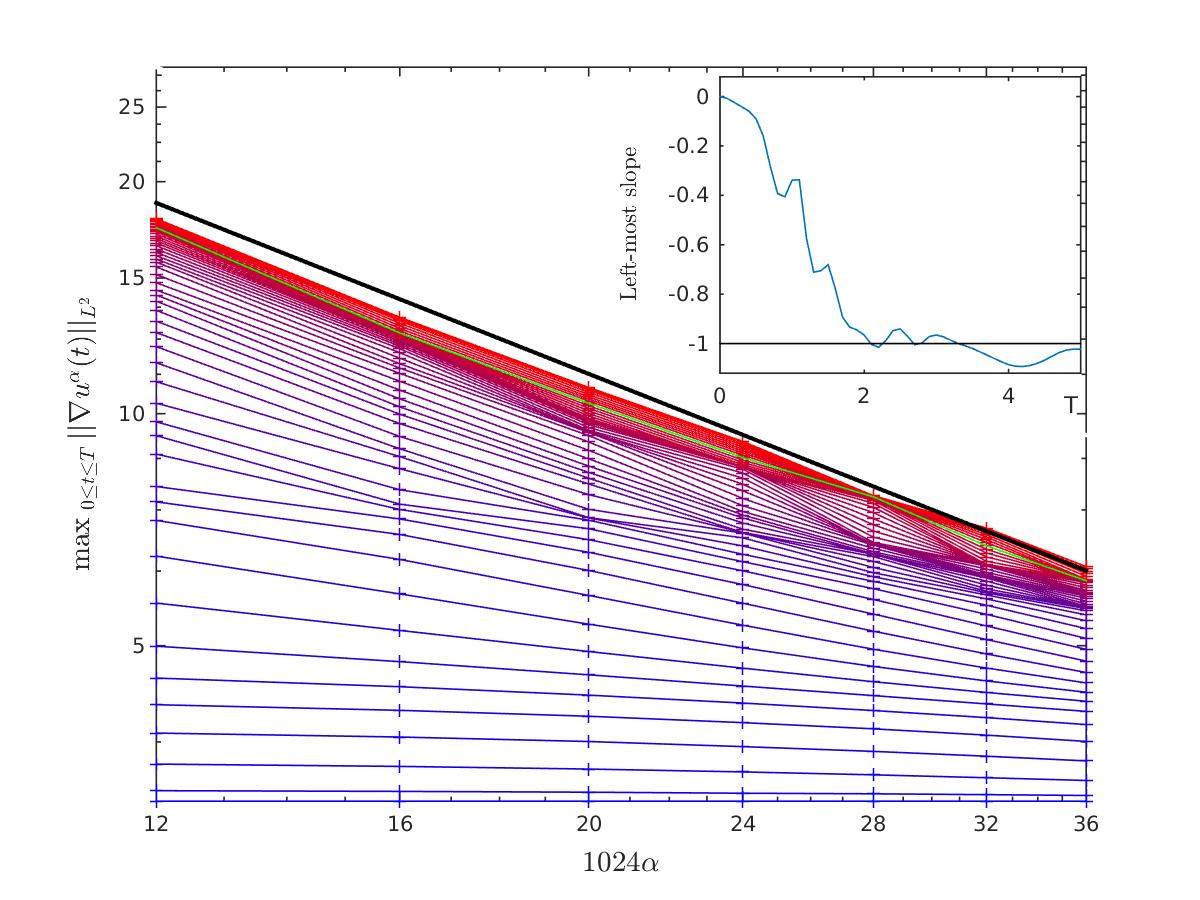}

\vspace{-0.3cm}

\caption{Log-log plot of $\max_{t\in[0,T^*]}\|\nabla\bu^\alpha(t)\|_{L^2}$ vs. $\alpha$ for the 3D Euler-Voigt equations at $T = 0.0, 0.1, \ldots, 4.9, 5.0$, $\alpha=12/1024,\ldots,36/1024$.   The thick black line is $C\alpha^{-1}$ vs. $\alpha$.  Green curve corresponds to $T = 4.2$. Resolution: $N^3=1024^3$ for $\alpha\leq 24/1024$, $N^3=512^3$ $\alpha\geq 28/1024$. Inset: Slope between $\alpha=12/1024$ and $\alpha=16/1024$.   Minimum value of $-1.0931$ at $T=4.2$.
}\label{fig_EV_max_blow_up}
\end{figure}


\section{Blow-Up for Burgers via the Benjamin-Bona-Mahony Equations}\label{sec_BBM}
In this section, we consider the 1D Benjamin-Bona-Mahony (BBM) equation for water waves, given by
\begin{align}\label{BBM}
  -\alpha^2 u_{txx} + u_t +uu_x = 0, \qquad u(x,0) = u_0(x).
\end{align}
This equation was derived in \cite{Benjamin_Bona_Mahony_1972} as a model for 
water waves, where it was shown to be globally well-posed.  It can be viewed as a regularization of the inviscid Burgers equation by formally setting $\alpha=0$ in \eqref{BBM}.  Notably, we do not propose here that the solution of \eqref{BBM} converges to the unique entropy solution of Burgers equation. We view this equation as a 1D analogue of the Euler-Voigt equations, with a crucial difference being that the pressure and the divergence-free condition are absent.   One advantage of considering equation \eqref{BBM} is that that solutions to the Burgers equation are known to develop a singularity in finite time; a fact that is unknown for solutions of the 3D Euler equations.  By following arguments similar to those in \cite{Larios_Titi_2009}, it is straight-forward to show that the analogue of \eqref{blow_up_max_first} implies blow-up for the Burgers equation on $[0,T^*]$.  

%
%

We use the method described in Section \ref{sec_blow_up} to try to identify the known singularity in Burgers equation ($u_t + uu_x = 0$).   That is, we test the analogue of criterion  \eqref{blow_up_max_first} for problem \eqref{BBM}, as $\alpha\maps0$.  The domain is the periodic interval $[-\pi,\pi]$, and the initial data is $u_0(x) = -\sin(x)$.  The solution of Burgers equation  with this initial data develops a singularity at time $T^*=1$.  

%
%

Fig. \ref{fig_BBM_max_blow_up} is analogous to Fig. \ref{fig_BBM_max_blow_up_nu}.
In Fig. \ref{fig_BBM_max_blow_up}, before the (Burgers) blow-up time $T^*=1$, the curves tend to decay faster than $\alpha$ as $\alpha\maps0$.  However, slightly after $T=1.0$, the curves become slightly  convex on the log-log plot for small $\alpha$.  If this trend continues as $\alpha\maps0$, 
the analogue of criterion \eqref{blow_up_max_first} implies Burgers equation develops a singularity at or before time $T^*\approx1.138$. This is already known by other means (e.g., the method of characteristics), but the results here serve to corroborate criterion \eqref{blow_up_max_first} as a test for blow-up. 

Finally, we repeat the simulation carried out to generate Fig. \ref{fig_BBM_max_blow_up}, except that we use the viscous BBM equation ($\nu=0.005>0$) instead of equation \eqref{BBM}.  Namely, we consider
\begin{align}\label{BBM_nu}
  -\alpha^2 u_{txx} + u_t +uu_x = \nu u_{xx}.
\end{align}
Due to the well-known fact that the viscous Burgers equation ($u_t +uu_x = \nu u_{xx}$) does not develop a singularity, we expect that criterion \eqref{blow_up_max_first} will not detect a singularity.  Indeed, in Fig. \ref{fig_BBM_max_blow_up_nu} we see that the curves do not obtain the critical slope value of $p=-1$ as $\alpha\maps0$, and indeed the lowest value is $\approx-0.235$, far away from the critical value.  Thus, in the case of Burgers equation, criterion \ref{blow_up_max_first} detects a singularity in the inviscid case, and does not detect one in the viscous case, exactly as expected.

\begin{figure}[htp]
\includegraphics[width=0.9\columnwidth, clip, trim=0mm 0mm 0mm 0mm]{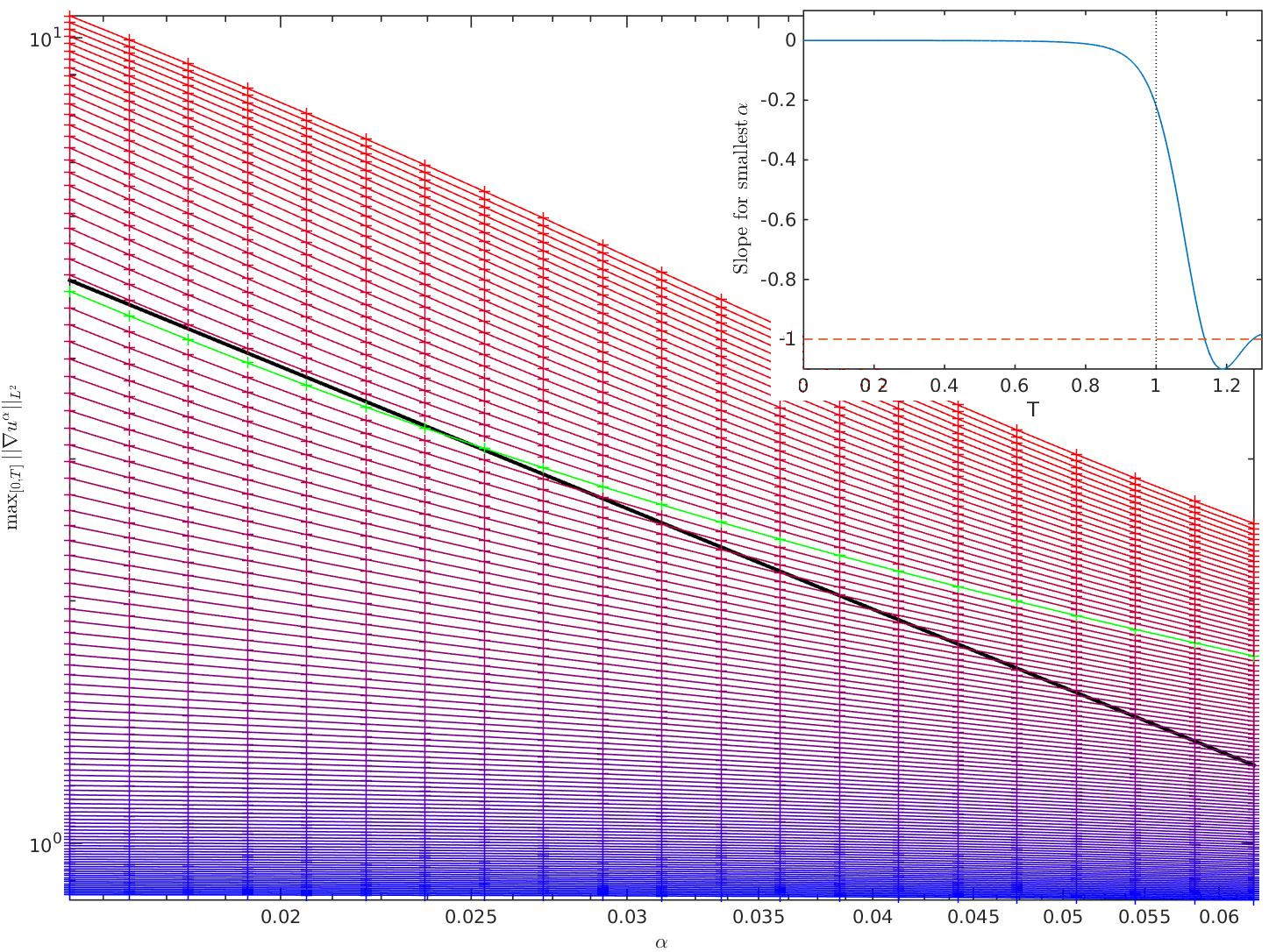}


\caption{Log-log plot of $\max_{0\leq t\leq T}\|u^\alpha_x(t)\|_{L^2}$ vs. $\alpha$ or the \textbf{inviscid} ($\nu=0$) BBM equations at various values of $T=0.65,\ldots,1.25$.  Green curve corresponds to $T\approx1.138$. 
Inset: Slope near smallest $\alpha$-values drops below -1 at $T\approx1.138$, indicating a blow-up at or before this time.
Resolution: $N = 8192$.}\label{fig_BBM_max_blow_up}
\end{figure}

\begin{figure}[htp]
\includegraphics[width=0.9\columnwidth,clip,trim=0mm 0mm 0mm 0mm]{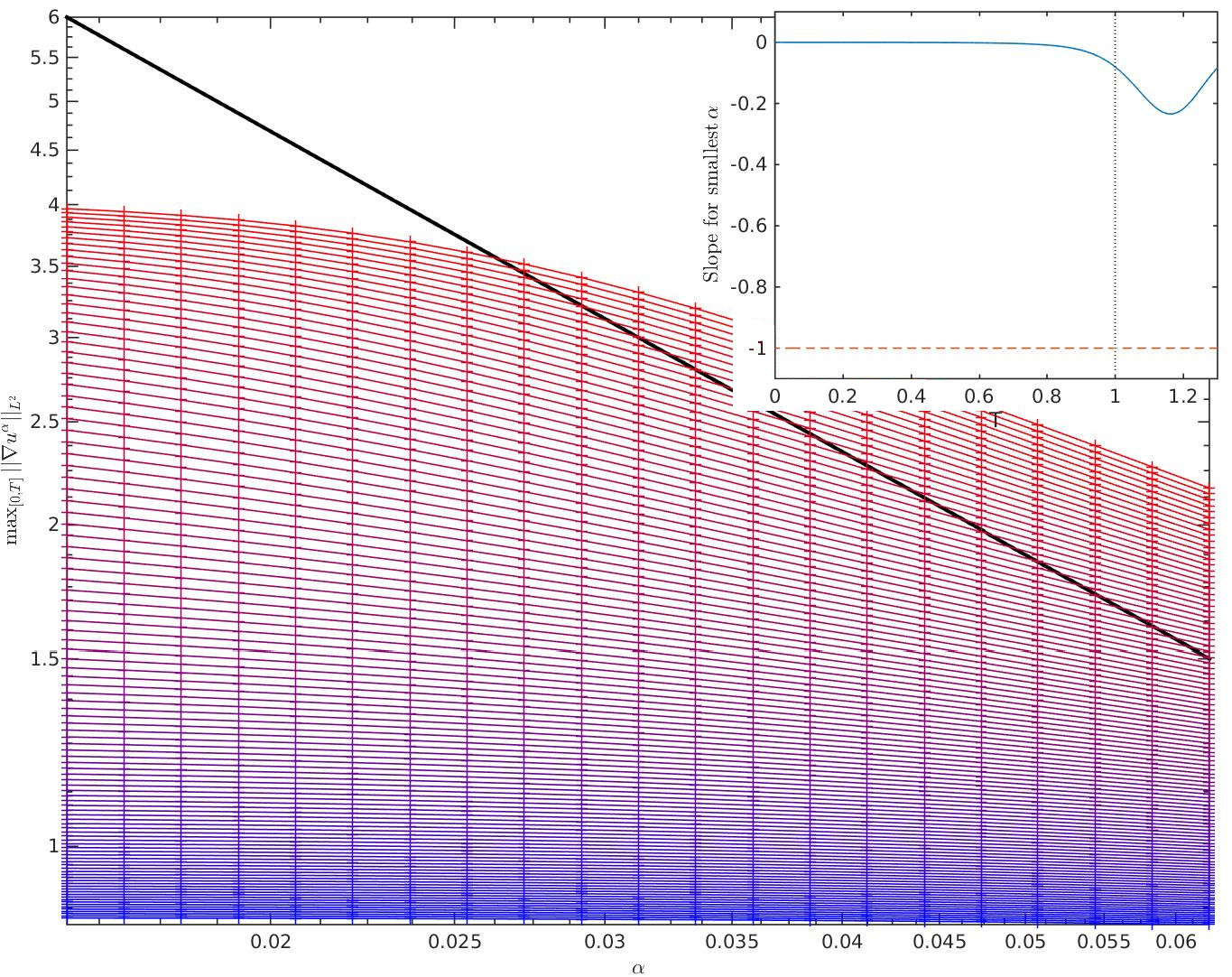}


\caption{Log-log plot of $\max_{0\leq t\leq T}\|u^\alpha_x(t)\|_{L^2}$ vs. $\alpha$ for the \textbf{viscous} ($\nu=0.005>0$) BBM equations at various values of T. Same $T$ values as in Fig. \ref{fig_BBM_max_blow_up}. 
Inset: Slope never drops below -1, meaning no blow-up is detected. Resolution: $N = 8192$.
}\label{fig_BBM_max_blow_up_nu}
\end{figure}

Finally, for two fixed values of $\alpha$, namely $\alpha_1 = 128/8192$ and $\alpha_2 = 138/8192$, we compute the value of the minimum slope as $\nu\maps0$; that is, \[S_{\text{min}}(\nu):=\min_{0<t<T}(\|u^{\alpha_2}_x(t)\|_{L^2} - \|u^{\alpha_1}_x(t)\|_{L^2})/(\alpha_2 - \alpha_1)\] as $\nu\maps0$, where $u^{\alpha_1}$ and $u^{\alpha_1}$ are solutions to \eqref{BBM_nu}.  This idea was suggested to us by one of the reviewers.  It demonstrates the dependence of the blow-up quantity on $\nu$, at least for a given resolution.  One can see a smooth transition from right to left as $\nu\rightarrow0$, crossing the blow-up criterion value of $-1$ roughly at viscosity $\nu_*=2.3\times10^{-4}$.  Since Burgers equation is globally well-posed for any $\nu>0$, for $0<\nu\ll 2.3\times10^{-4}$, the detection yields a false positive for singularity formation here.  This underscores the need for higher-resolution studies (which would allow for smaller $\alpha$-values), as well as enhanced extrapolation methods.
\begin{figure}[htp]
\includegraphics[width=0.75\columnwidth,clip,trim=0mm 0mm 0mm 0mm]{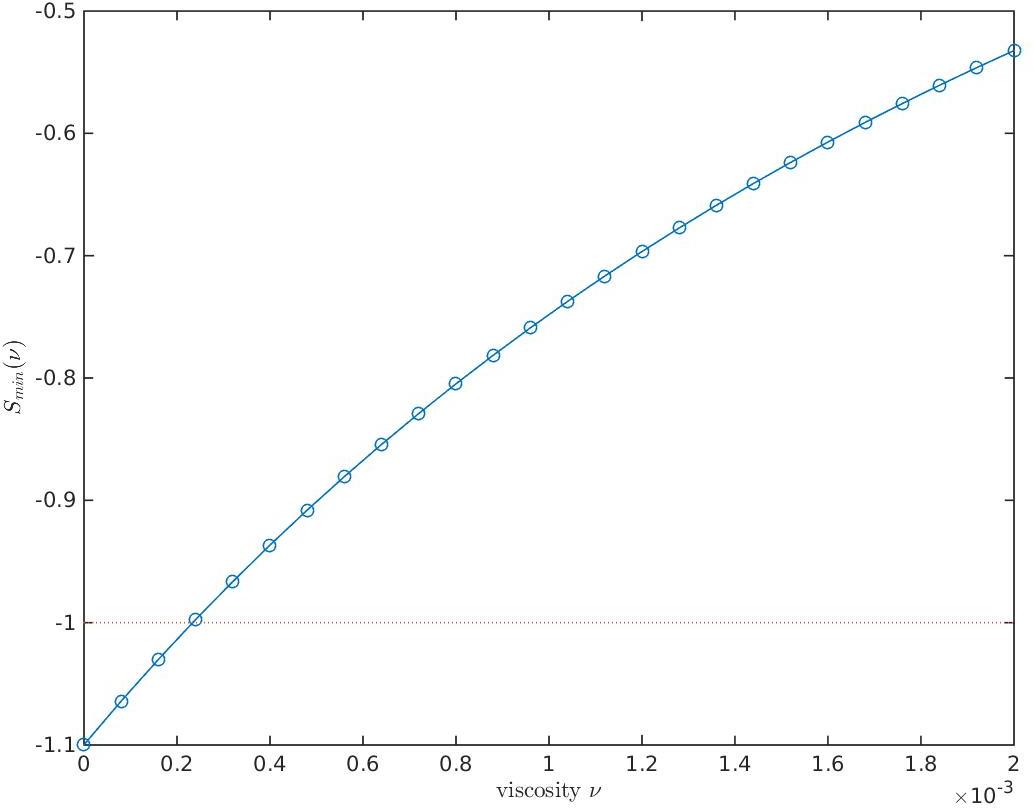}
\caption{The quantity $S_{\text{min}}$ vs. $\nu$. This shows the dependence on the minimal slope of the blow-up quantity for values $\alpha_1 = 128/8192$ and $\alpha_2 = 138/8192$.  Resolution 8192.}\label{fig_min_slope}
\end{figure}

\section{Conclusion}\label{sec_conclusion}
The results in Section \ref{sec_blow_up} provide computational evidence for the development of a singularity of the 3D Euler equations with Taylor-Green initial data \eqref{TaylorGreen}, near time $T=4.2$.  
Future studies at smaller $\alpha$-values (and thus higher 
resolution), combined with state-of-the-art extrapolation methods, may either 
corroborate or contradict these findings.
In any case, the approach presented here represents a new method in the computational search for singularities, and its effectiveness has been demonstrated in the case of Burgers equation.

\section*{Acknowledgments}

The work of E.S.T. was supported in part by ONR grant number N00014-15-1-2333, and by the NSF grants number DMS-1109640 and DMS-1109645. 



\begin{scriptsize}

\end{scriptsize}

\end{document}